\documentclass[a4paper, 11pt]{article}

\usepackage{url}
\usepackage{mathrsfs}

\usepackage{amsfonts,amsmath,amssymb}
\usepackage{amsthm}
\usepackage{enumerate}
\usepackage{a4wide}
\usepackage[all]{xy}

\theoremstyle{plain}
\newtheorem{proposition}{Proposition}[section]
\newtheorem{theorem}[proposition]{Theorem}
\newtheorem*{theoremst}{Theorem}
\newtheorem{lemma}[proposition]{Lemma}
\newtheorem{corollary}[proposition]{Corollary}
\newtheorem*{corollaryst}{Corollary}

\theoremstyle{definition}
\newtheorem{definition}[proposition]{Definition}

\theoremstyle{remark}
\newtheorem{remark}[proposition]{Remark}
\newtheorem{example}[proposition]{Example}

\everymath{\displaystyle}

\DeclareMathOperator{\Supp}{Supp}

\DeclareMathOperator{\Conv}{Conv}
\DeclareMathOperator{\Aut}{Aut}
\DeclareMathOperator{\proj}{proj}

\DeclareMathOperator{\Haar}{Haar}
\newcommand{\tld}{\widetilde}
\DeclareMathOperator{\Csph}{\mathscr C_{\mathrm {sph}}}
\DeclareMathOperator{\Rsph}{\mathrm {Res}_{\mathrm {sph}}}

\newcommand{\eps}{\varepsilon}

\newcommand{\C}{\mathbb{C}}                          % |C
\newcommand{\N}{\mathbb{N}}                          % |N
\newcommand{\Z}{\mathbb{Z}}                          % /Z
\newcommand{\PR}{\mathbf{P}}

\author{Jean L\'ecureux\footnote{Universit{\'e} de Lyon;
Universit{\'e} Lyon 1;
INSA de Lyon;
Ecole Centrale de Lyon;
CNRS, UMR5208, Institut Camille Jordan,
43 blvd du 11 novembre 1918,
F-69622 Villeurbanne-Cedex, France.  lecureux@math.univ-lyon1.fr}}
\date{}
\begin{document}
\title{Amenability of actions on the boundary of a building }
\maketitle
\begin{abstract}
We prove that the action of the automorphism group of a building on its boundary is topologically amenable. The notion of boundary we use was defined in a previous paper \cite{CL}. It follows from this result that such groups have property (A), and thus satisfy the Novikov conjecture. It may also lead to applications in rigidity theory. 
\end{abstract}

\section*{Introduction}
The notion of an amenable action was first defined by R. Zimmer \cite{Zim}. He generalized the notion of amenability for a group \textit{via} a fixed point property: his definition of an amenable action on some space $S$ implies some fixed point property for a bundle  of compact convex subspaces over $S$. 

One of the interests in introducing this notion is to explain and generalize the existence of equivariant maps from $S$ to the set of probability measures on any compact metric $G$-space. These maps were used by Furstenberg and then Margulis to prove rigidity theorems \cite{Mar}.

Amenability of group actions has been intensively studied since then, sometimes in other contexts, for example in the study of bounded cohomology \cite{Mon} or of $C^*$-algebras \cite{HR}. As in the case of amenability of groups, there are many definitions of amenability of actions. Unfortunately, not all of them are equivalent. The notion we use, which we recall in Section \ref{amenability}, is a topological notion, whereas Zimmer's notion is a measured one. Namely, we call an action of a group $G$ on a space $X$ \emph{amenable} if there exists a sequence of continuous, almost equivariant maps from $X$ to $\PR(G)$, where $\PR(G)$ is the set of probability measures in $G$ -- see Section \ref{amenability} for a more precise definition. This topological notion of amenability is stronger than amenability in the sense of Zimmer.

The argument which was first used to prove that some actions are amenable is the following: if $P$ is an amenable subgroup of some group $G$, then $G/P$ is an amenable $G$-space \cite[Proposition 4.3.2]{Zimmer}. In particular, this is true when $G$ is a real semisimple Lie group and $P$ a minimal parabolic subgroup of $G$. However, the question of knowing whether an action is amenable is more delicate when the space acted upon is not homogeneous.\\

Buildings are combinatorial and geometric objects introduced by J. Tits. We recall their definitions and some basic properties in Section \ref{buildings}. Tits' aim was to describe the structure of semisimple algebraic  groups; he proved that these groups act on spherical buildings. These buildings are constructed by gluing tesselations of spheres.
Later on, Bruhat and Tits \cite{BT,BT2} proved that, in the case of a semisimple group over a local field, this construction could be generalized: to such a group is associated an affine building, constructed by gluing  tesselations of Euclidean planes.

However, there are many other types of buildings. Even for affine buildings, there are some constructions of buildings that do not come from an algebraic group \cite{Bar,CMSZ2,Ron}. Some of these buildings have a cocompact automorphism group, some have no automorphism at all. One can also construct buildings from many types of Coxeter groups, instead of just spherical and affine ones. For example, there are some constructions of Fuchsian buildings \cite{GaP,Bou}, which are gluing of tesselations of the hyperbolic plane. Some other buildings with non-affine tessellations come from the so-called Kac-Moody groups \cite{Ti}.

It is an interesting question to know how close these groups are from classical algebraic groups. It has been proven that they share some properties. For example, the article \cite{Bou} proves some Mostow rigidity for groups acting on Fuchsian buildings. Kac-Moody groups have been proven to have the Normal Subgroup Property (any non-trivial normal subgroup is of finite index, or is finite and central) \cite{BS06}. However, Kac-Moody groups are not linear \cite{CR}, which makes them of course quite different from algebraic groups.

In the case of a $p$-adic semisimple Lie group $G$, if $P$ is a minimal subgroup, we know that the action of $G$ on $G/P$ is amenable. In this setting, the set $G/P$ has a combinatorial interpretation: it is the set of chambers of the spherical building at infinity of the affine building associated to $G$: an element of $G/P$ can be seen as an equivalence class of sectors \cite[11.8]{AB}.
 The construction of this building at infinity can be described purely in terms of the geometry of the building, and therefore also makes sense for affine buildings that do not come from a $p$-adic group. It is therefore natural to wonder whether the action of the automorphism group of such a building on the set of chambers of the spherical building at infinity is amenable or not. 
In the case of buildings of type $\widetilde A_2$, this  has been proved to be true in \cite[section 4.2]{RS}.

For arbitrary, non necessarily affine buildings, it is not really obvious to know which boundary should be considered. One could for example consider the visual boundary of a building, seen as a CAT(0) space. However, this boundary does not have nice combinatorial properties, as in the affine case, and the action of the group on the visual boundary is not amenable. In a previous paper \cite{CL}, we defined a notion of combinatorial compactification, denoted $\Csph(X)$, for any building $X$. This construction generalizes a previous construction for affine buildings \cite{GuR}. In the affine case, this compactification contains not only the equivalence classes of sectors, but also equivalence classes of ``chemin\'ees'' (``chimney'' in english), defined in \cite{Rou}, which are some kind of half-strips in apartments. In particular, we generalize the notion of sector (including the chemin\'ees) to arbitrary buildings.

In this context, we prove:

\begin{theoremst}\label{moyen}
Let $X$ be any building, and $G$ a locally compact subgroup of $\Aut(X)$ acting properly on $X$. Then the action of $G$ on the combinatorial compactification of $X$ is amenable.
\end{theoremst}

 In particular, we prove that closed subgroups of reductive groups over local fields have an amenable action on this combinatorial boundary (which, as explained above, contains $G/P$ as a closed subset). Similarly, Kac-Moody groups, which act on a product of buildings $X_+\times X_-$, and their complete versions \cite{RR}, which act on only one of these two buildings, admit an amenable action on this combinatorial boundary.

From the point of view of the study of a group, the existence of some amenable actions has many applications. First of all, the existence of some compact space $X$ with an amenable action of a discrete group $G$ implies a metric property on $G$, namely Property (A) \cite{HR}. This property was proved in \cite{Yu} to imply that there exists an uniform embedding of $G$ into a Hilbert space. He also proves that this in turn implies the coarse Baumes-Connes conjecture and the Novikov conjecture for $G$.  As a consequence of our result, we get the following corollary, valid for every building:

\begin{corollaryst}
 Let $X$ be a locally finite building and $G=\Aut(X)$. Let $\Gamma$ be a closed, discrete subgroup of $G$. Then $\Gamma$ satisfies the Novikov conjecture.
\end{corollaryst}

This corollary can also be deduced by combining \cite{KaS}, together with the fact that  buildings  are complete, geodesic and bolic spaces \cite[Corollary 8]{BKa}.

These applications only require the existence of some compact set $K$ on which $G$ acts amenably. There are some other kinds of maps in which the nature of $K$ may be useful. For example, \cite[Theorem 5.7.1]{Mon} allows us to find some resolution of a coefficient Banach module $E$ which (at least theoretically) calculates the continuous bounded cohomology of $G$ with values in $E$.

 The amenability of some actions is also a crucial point in the proof of superrigidity theorems \`a la Margulis \cite{Mar,ACB,Zimmer}. They imply the existence of ``boundary maps''\cite[Proposition 4.3.9]{Zimmer}. Thus, we get:

\begin{corollaryst}
 Let $\Gamma$ be a closed subgroup of the automorphisms of $X$. Let $K$ be a compact metric $\Gamma$-space. Then there exists a $\Gamma$-equivariant map from $\Csph(X)$ to the set $\PR(K)$ of probability measures on $K$.
\end{corollaryst}

 In the classical context of a real Lie group $G$, if $\Gamma$ is a lattice in $G$, acting on a vector space $V$, there is a $\Gamma$-equivariant map from $G/P$ to $\PR(V)$. This boundary map is used, together with some ergodicity properties, to extend the action from $\Gamma$ to $G$, and thus prove a rigidity theorem. The same kind of argument, using the existence of a boundary map, together with some ergodicity properties, was used afterwards in some other situations, for example for groups acting on CAT(-1) spaces \cite{BM}. We hope that such a line of arguments, with an explicit source for the boundary map, could also lead in our case to rigidity and non-linearity properties for lattices of arbitrary buildings \cite{RGAF}. This could also lead to superrigidity results for lattices of irreducible buildings, a complementary case with respect to \cite{Monod}, applied to products of buildings.\\

Let us finigh this introduction by presenting shortly some previous results. The article \cite{Ja} proves that Coxeter groups are amenable at infinity: there exists some compact subspace on which they act amenably. The proof consists essentially in the remark that a Coxeter group is metrically a subset of a finite product of homogeneous trees, remark which we will use again here. The advantage of our proof, even in the case of Coxeter group, is that we have an explicit description on the compact set they act on.
For hyperbolic spaces with a cocompact action, amenability of the action on the (visual) boundary is well-known. The first proof is due to S. Adams \cite{A}. Another proof can be found in \cite{Ka}. 

The strategy we use is the following. First, we use Kaimanovich's proof to get something a little more precise than amenability for the action of a group on a tree. Then we use the remark of \cite{Ja}  that Coxeter groups are embedded in product of trees, and we prove that the action on the combinatorial boundary of product of trees is also amenable. Thus, the action of a Coxeter group on its combinatorial boundary is also amenable. Hence, we get a construction of measures on Coxeter complexes. Then, using incidence properties of the combinatorial sectors, we glue these measures (defined on apartments) to get measures on the combinatorial boundary of the whole building.\\

{\bf Structure of the paper.} The first section shortly recalls the definition and some basic facts about buildings. The second section explains the combinatorial boundary of buildings. Then, the third section is about the embedding of a Coxeter complex into a product of trees, and how this embedding behaves with respect to the compactification. The fourth section recalls the various definition of amenability, and of course especially the one we use. The fifth section proves that Coxeter complexes are boundary amenable, in such a way that the measures we obtain can, in the last section, be glued together to prove the boundary amenability of a group acting on a building.

{\bf Acknowledgment.} The author is very grateful to Bertrand R\'emy for his constant advice.

\section{Buildings}\label{buildings}

In this section, we recall the definition and some standard facts about buildings. Standard books on this subject are \cite{AB} and \cite{Ro}.

There are several  points of view on buildings, some of them geometric, and others more combinatorial. In view of the definition of the combinatorial boundary \cite{CL}, we adopt a combinatorial point of view. Namely, we see a building as a $W$-metric space, where $W$ is a Coxeter group \cite{Ti}.

\subsection{Coxeter groups}
We recall some basic facts about Coxeter groups. All of this is present in \cite[IV.§1]{BBK}.
Recall that a \emph{Coxeter group} is a group $W$ given by a presentation $$W=\langle s_1,\dots,s_n \;|\; (s_is_j)^{m_{ij}}\rangle,$$
where $m_{ij}\in\N\cup\{\infty\}$ are such that $m_{ii}=1$ and $m_{ij}=m_{ji}$ (by $m_{ij}=\infty$, we mean that there is no relation between $s_i$ and $s_j$). The generating set $S=\{s_1,\dots,s_n\}$ is  a part of the data. 

The group $W$ should be thought of as a group generated by a set of reflections $S$. Indeed, standard examples of Coxeter groups are discrete reflection subgroups of a Euclidean space or a hyperbolic space. In general, we still use the word  reflections:

\begin{definition}
 A \emph{reflection} in $W$ is an element of $W$ which is conjugated to a canonical generator of $S$.
\end{definition}

In the sequel, we denote by $(W,S)$ a Coxeter group. An element $w$ of $W$ can be written as a word in $S$. The length $l(w)$ of $w$ is defined as the minimal length of such words. Recall the following basic fact:  if $s\in S$ and $w\in W$, then either $l(sw)=l(w)+1$ or $l(sw)=l(w)-1$.

\begin{definition} The \emph{Coxeter complex} $\Sigma$ associated to $W$ is the Cayley graph of $W$ with respect to the set of generators $S$. 
\end{definition}

Thus, the set of vertices of $\Sigma$ is canonically in bijection with $W$, and each edge of $\Sigma$ is labelled by an element of $S$. Left translation by elements of $W$ endowes $\Sigma$ with a left $W$-action. The labelling of edges is $W$-equivariant.

 In the langage of buildings, vertices of $\Sigma$ are called \emph{chambers} and edges of $\Sigma$ are called \emph{panels}. If $C_1$ and $C_2$ are chambers in $\Sigma$, a \emph{gallery} between $C_1$ and $C_2$ is a path in $\Sigma$ from $C_1$ to $C_2$.

Since edges are labelled by elements of $S$, such a gallery gives rise to a word with letters in $S$, and therefore to an element of $W$. If $C_1=1_W$, by definition of $\Sigma$, this element is exactly $C_2$, and thus does not depend on the gallery chosen. By $W$-equivariance, it is the same for general chambers $C_1$ and $C_2$. We define the \emph{Weyl distance}, or $W$-distance $\delta(C_1,C_2)$ between $C_1$ and $C_2$ to be the element of $W$ discussed above. Note that, seeing $C_1$ and $C_2$ as elements of $W$, we have also $\delta(C_1,C_2)=C_1^{-1}C_2$. Note also that the distance between $C_1$ and $C_2$ in $\Sigma$ exactly the length $l(\delta(C_1,C_2))$.

\subsection{Buildings and apartments}
Let $W$ be a Coxeter group with generating set $S$. We assume $S$ to be finite. In this section, we define buildings of type $(W,S)$.

Coxeter complexes are  special, and perhaps the most simple, cases of buildings. As we have seen, they are endowed with a Weyl distance. A building can be seen as a set endowed with this kind of $W$-distance. More precisely \cite[Definition 5.1]{AB}:

\begin{definition}
 A \emph{building} is a set $\mathcal C$, whose elements are called \emph{chambers}, endowed with a map $\delta : \mathcal C\times \mathcal C\to W$, called the \emph{Weyl distance}, such that:
\begin{enumerate}[(i)]
 \item $\delta(C,D)=1$ if and only if $C=D$
 \item If $\delta(C,D)=w$ and $C'\in \mathcal C$ is such that $\delta(C',C)=s$, then $\delta(C',D)\in\{sw,w\}$. If furthermore $l(sw)=l(w)+1$ then $\delta(C',D)=sw$.
\item If $\delta(C,D)=w$ then for any $s\in S$ there is a chamber $C'$ such that $\delta(C',C)=s$ and $\delta(C',D)=sw$.
\end{enumerate}
\end{definition}

\begin{remark}
A building can alternatively be seen as a higher-dimensionnal simplicial complex of  or as a metric space of nonpositive curvature. But the combinatorics from all these points of view, including ours, is exactly the same: the only difference is that from our point of view, a chamber is just a point; from a simplicial point of view, a chamber  is a simplex; and from a metric point of view, a chamber is some (well-chosen) compact metric space.
\end{remark}

\begin{example}
 The Coxeter complex $\Sigma$ associated to $(W,S)$ is a building. In this example, in axiom (ii), we always have in fact  $\delta(C',D)=sw$.
\end{example}

\begin{example}
 Let $(X_1,\delta_1)$ be a building of type $(W_1,S_1)$ and $(X_2,\delta_2)$ be a building of type $(W_2,S_2)$. Then $X_1\times X_2$, endowed with the Weyl distance $\delta_1\times\delta_2$, is  a building of type $(W_1\times W_2,S_1\cup S_2)$.
\end{example}

\begin{definition}
 Let $X$ and $X'$ be buildings of type $(W,S)$, with associated Weyl distances $\delta$ and $\delta'$ respectively. A \emph{$W$-isometry} $f$ between some subsets $A\subset X$ and $B\subset X'$ is a map $f:A\to B$ such that for every $C,D\in A$, we have $\delta(C,D)=\delta'(f(C),f(D))$. An \emph{automorphism}, or more precisely \emph{type-preserving automorphism}, of $X$ is a bijective $W$-isometry from $X$ to $X$.
\end{definition}

From a building $X$, we can produce a graph by adding an edge between two chambers $C$ and $C'$ if $\delta(C,C')=s\in S$, and we label such an edge by $s$. We also call this graph a building. Two chambers at distance $1$ in this graph will be called \emph{adjacent}. A path in this graph is called a \emph{gallery}. As in the case of Coxeter complexes, $\delta(C,D)$ can be read as the label of the edges of any minimal gallery between $C$ and $D$. 

If $C$ is a chamber and $s\in S$, the set of chambers $C'$ such that $\delta(C,C')=s$ is called a \emph{panel}. Any chamber which belongs to a panel $\sigma$ is said to be \emph{adjacent} to $\sigma$. The \emph{thickness} of a panel is its cardinality. If all panels have finite thickness, the building $X$ is said to be \emph{locally finite}.

\begin{definition}
 Let $X$ be a building. A \emph{convex} or \emph{combinatorially convex} subset of $X$ is a subset $A$ such that any minimal gallery between chambers in $A$ is contained in $A$.
\end{definition}

We can now define apartments in a building $X$:

\begin{definition}
 An apartment $A$ in a building $X$ is a subset of $X$ which is $W$-isometric to the Coxeter complex $\Sigma$ associated to $W$.
\end{definition}

With this definition, it is not really easy to see that apartments exist in a given building. The following theorem, first proved in \cite[§3.7.4]{Tit81} (see also \cite[Theorem 5.73]{AB}), ensures there are many ones:

\begin{theorem}\label{lemTits}
 Let $X$ be a building of type $(W,S)$. Then any subset of $X$ which is $W$-isometric to a subset of $\Sigma$ is contained in an apartment.
\end{theorem}

One should think of apartments as ``slices'' of the building. Properties (ii) and (iii) in the following proposition are often given as another definition of buildings.

\begin{proposition}
 Let $X$ be a building of type $(W,S)$.
\begin{enumerate}[(i)]
 \item Every apartment in $X$ is convex.
\item For every two chambers $C$ and $D$ in $X$, there exists an apartment $A$ in $X$ such that $C,D\in A$.
\item  Let $A_1$ and $A_2$ be two apartments in $X$. If $C\in A_1\cap A_2$, then there exists a bijective $W$-isometry $\phi:A_1\to A_2$ such that $\phi(C)=C$.
 \end{enumerate}

\end{proposition}
\begin{proof}
 This is Corollary 5.54, Corollary 5.68 and Corollary 5.74 in \cite{AB}. Note that (ii) is a direct consequence of Theorem \ref{lemTits}, since $\{C,D\}$ is $W$-isometric to $\{1_W,\delta(C,D)\}\subset \Sigma$.
\end{proof}

\begin{example}
 A tree without endpoints is a building. Here the Weyl group is the infinite dihedral group $D_\infty$ ; apartments are bi-infinite lines \cite[Proposition 4.44]{AB}. More precisely, following our definition, the building is the set of edges of this tree.

The set of vertices of a regular tree of valency $k$ can also be considered as a building, or more precisely, as a Coxeter complex. Its Weyl group is the free product of $k$ copies of $\Z/2\Z$.
\end{example}

\begin{example}
 Let $k$ be a field, and $V$ a vector space of dimension $n$ over $k$. The set of maximal flags of subspaces of $V$ is a building \cite[4.3]{AB}. The Weyl group is the permutation group over $n$ indices. 
\end{example}

\begin{example}
 Let $F$ be a local, nonarchimedean field. Let $G$ be a semisimple algebraic group defined over $k$. Bruhat-Tits theory \cite{BT} associates to such a $G$ a building. The Weyl group of this building is a Euclidean reflection group: it is a group generated by reflections with respect to a locally finite set of hyperplanes, in a Euclidean space of dimension $r$, where $r$ is the rank of $G$.
\end{example}

\subsection{Projections and roots}
\begin{proposition}
Let $C$ be a chamber, and let $\sigma$ be a panel. There exists a unique chamber $C'$ adjacent to $\sigma$ at minimal distance from $C$.
\end{proposition}
\begin{proof}
This is a particular case of \cite[Proposition 5.34]{AB}.
\end{proof}

\begin{definition}
The chamber $C'$ defined as above is called the \emph{projection} of $C$ onto $\sigma$, and is denoted $\proj_\sigma(C)$.
\end{definition}

\begin{definition}
The set of chambers of $\Sigma$ whose projection on a given panel is constant is called a \emph{root} of $\Sigma$, or a \emph{half-apartment}. More generally, a root (or a half-apartment) in a building $X$ is a subset of $X$ which is $W$-isometric to a root of $\Sigma$. 
\end{definition}

Let us restrict now to the case when $X$ is the Coxeter complex $\Sigma$ of $W$. 
Consider $1_W\in \Sigma$ and the panel $\sigma$ of type $s$ adjacent to $1_W$. The root $\alpha_s$ consisting of the chambers whose projection on $\sigma$ is $1_W$ can be identified with $\{w\in W\;|\;l(sw)>l(w)\}$. All roots of $\Sigma$ are conjugated by some element of $W$ to a root of this type. Thus there is only a finite number of orbits of roots in $\Sigma$.

Note also that, in $\Sigma$, every panel is adjacent to exactly two chambers. Let $\sigma$ be a panel in $\Sigma$, and let $x$ and $y$ be the chambers adjacent to $\sigma$. Then $\Sigma$ is the union of two disjoint roots: the set of chambers $C$ such that $\proj_\sigma(C)=x$, and the set of chambers $C'$ such that $\proj_\sigma(C')=y$. These two roots are called \emph{opposite}. If $\alpha$ is a root of $\Sigma$, its opposite root is usually written $(-\alpha)$. In the case of a root of the form $\alpha_s=\{w\in W\;|\;l(sw)>l(w)\}$, we can see that the reflection $s$ sends $\alpha_s$ to $(-\alpha_s)$. By conjugating by an element of $W$, in the case of a general root $\alpha$, we can see that there is always a reflection of $W$ which switches $\alpha$ and $(-\alpha)$. We call this reflection the \emph{reflection associated to the wall of $\alpha$}.

\begin{definition}
 Let $\alpha$ be a root of $\Sigma$. The set of panels which are adjacent to some chamber of $\alpha$ and to some chamber of $(-\alpha)$ is called the \emph{wall} associated to $\alpha$.
\end{definition}

Note that the wall of $\alpha$ divides the graph $\Sigma$ into two connected components, which are $\alpha$ and $(-\alpha)$.

\subsection{Residues}
There is a natural generalization of the notion of panels, which is the notion of \emph{spherical residue}. This notion is useful to define the combinatorial compactification. In the following, $X$ is a building of type $(W,S)$.

\begin{definition}
 Let $J\subset S$, and $W_J$ the subgroup of $W$ generated by $J$. Let $C$ be a chamber in $X$. The set $R$ of chambers $C'\in X$ such that $\delta(C,C')\in W_J$ is called a \emph{residue of type $J$} of $X$.
If $W_J$ is finite, $R$ is called a \emph{spherical residue}.
\end{definition}

The main interest of residues is the following:
\begin{proposition}[{\cite[Corollary 5.30]{AB}}]
Let $R$ be a residue of type $J$ in the building $X$. Then $R$ is itself a building of type $W_J$.\begin{flushright}$\square$\end{flushright}
\end{proposition}

The set of spherical residues of $X$ is denoted by $\Rsph(X)$. Note that if the building $X$ is locally finite, then spherical residues are finite.

The next thing we do is to endow the set of spherical residues with the structure of a graph -- or, equivalently, to turn it into a discrete metric space. The most natural thing to do would be to consider the incidence graph of $\Rsph(X)$. However, the drawback of this point of view is that the graph of chambers of $X$ will not be isometrically embedded into $\Rsph(X)$. Hence we define another distance on $\Rsph(X)$.

By a slight abuse of language, we say that $R\in\Rsph(X)$ is contained in an apartment $A$ (respectively in a half-apartment $\alpha$) if $R$ contains a chamber which is in $A$ (respectively in $\alpha$).
\begin{definition}
 Let $R,S\in\Rsph(X)$, and $A$ an a apartment containing $R$ and $S$. Let $\Phi_A(R,S)$ (respectively $\Phi_A(S,R)$) be the set of roots of $A$ which contain $R$ but not $S$ (respectively, the set of roots which contain $R$ but not $S$). The \emph{root-distance} $d(R,S)$ between $R$ and $S$ is defined in the following way:

$$d(R,S)=\frac1 2 (|\Phi_A(R,S)|+|\Phi_A(S,R)|).$$
\end{definition}

It is quite easy to see that this root-distance does not depend on the choice of the apartment $A$. The fact that $d$ is indeed a distance needs some verification, see \cite[Proposition 1.2]{CL}.

\begin{example}
Let $W$ be the free product of $r$ copies of the group with two elements. The Coxeter complex $\Sigma$ of $W$ is a regular tree of valency $r$. Then the set of spherical residues of $\Sigma$ can be identified with the union of the set of vertices and the set of edges of $\Sigma$. In this case, the root-distance between two edges or two vertices is the usual distance. If $x$ is a vertex and $e$ an edge, the root distance $d(x,e)$ is equal to $1/2$ plus the minimal distance between $x$ and the extremal vertices of $e$.
\end{example}

\begin{example}
 Let $W=W_1\times W_2$, with Coxeter generators $S=S_1\cup S_2$. Let $X=X_1\times X_2$ be a building of type $W$, with $X_i$ of type $W_i$. Then the set of spherical residues of $X$ is equal to the product $\Rsph(X_1)\times \Rsph(X_2)$. Furthermore, this identification is isometric, when $\Rsph(X)$ and $\Rsph(X_i)$ (for $i=1,2$) are equipped with the root-distance.
\end{example}

As in the case of panels, there is a notion of projection on spherical residues:

\begin{definition}
 Let $R\in\Rsph(X)$ and $C$ a chamber in $X$. The \emph{projection} $\proj_R(C)$ of $C$ on $R$ is the unique chamber $C'\in R$ which realizes the minimum distance between $C$ and any chamber in $R$.

Let $S\in\Rsph(X)$. The projection $\proj_R(S)$ is the union of all the chambers $\proj_R(C)$, where $C\in S$; it is itself a spherical residue.
\end{definition}

\section{Combinatorial boundary}

Let $X$ be a building of type $(W,S)$. %We assume $X$ to be of finite thickness. 
In this section, we define the combinatorial compactification of $X$. All results, definitions and examples can be found in \cite{CL}.

\subsection{Construction of the boundary}
The projection of a spherical residue on another one defines a map 
$$\pi_{\mathrm {Res}}:\Rsph(X)\to\prod_{S\in\Rsph(X)} \Rsph(S), $$
which associates to a spherical residue $R$ the map $S\mapsto \proj_S(R)$.

We endowe the latter space with the product topology. If $X$ is locally finite, it is a compact space. It is quite easy to see \cite[Proposition 2.2]{CL} that $\pi_ {\mathrm {Res}}$ is $\Aut(X)$-equivariant, injective and has discrete image. Thus we can make the following definition:

\begin{definition}
 Let $\Csph(X)$ be the closure of $\pi_ {\mathrm {Res}}(\Rsph(X))$ in $\prod_{S\in\Rsph(X)} \Rsph(S) $. We call $\Csph(X)$ the \emph{combinatorial bordification} of $X$, or \emph{combinatorial compactification} if $X$ is locally finite (and hence $\Csph(X)$ is compact).
\end{definition}

\begin{example}
Let $T$ be a regular tree, seen either as a building of type $D_\infty$, or as an apartment for a free Coxeter group. Then the combinatorial compactification is the usual compactification: $\Csph(T)=\Rsph(T)\cup\partial_\infty T$.
\end{example}

\begin{example}
 Let $X_1$ and $X_2$ be buildings. As we have seen before, $\Rsph(X_1\times X_2)=\Rsph(X_1)\times \Rsph(X_2)$. Furthermore, if $S_i,R_i\in\Rsph(X_i)$ (for $i=1,2$), then we see that $$\proj_{(R_1,R_2)}(S_1,S_2)=(\proj_{R_1}(S_1),\proj_{R_2}(S_2)).$$
Thus, we have $\Csph(X_1\times X_2)=\Csph(X_1)\times \Csph(X_2)$. 
\end{example}

The combinatorial compactification is much more easily understood in restriction to apartments. A point $\xi$ in  the combinatorial compactification of a Coxeter complex $\Sigma$ is uniquely determined by the roots containing it. More precisely, there exists a set $\Phi(\xi)$ of roots of $\Sigma$ such that every sequence of spherical residues of $\Sigma$ converging to $\xi$ is eventually in every root of $\Phi(\xi)$. 

It is not at all obvious that such an analysis can be done for general buildings; for example, it is not obvious that any $\xi\in\Csph(X)$ is in the boundary of an apartment. Nevertheless, it is the case:

\begin{proposition}[{See \cite[Proposition 1.20]{CL}}]\label{phi}
For every $\xi \in U(X)$, there exists an apartment $A$ containing a sequence of spherical residues converging to $\xi$. For every such $A$, the point $\xi$ is uniquely determined by the set of roots $\Phi_A(\xi)$, which is the set of roots containing a sequence of residues which are all contained in $A$ and are converging to $\xi$ 
\end{proposition}

This proposition is quite useful to understand the compactification. For example, it is possible to understand the case of affine buildings:

\begin{example}
Let us consider an apartment $A$ of type $\tilde A_2$. It is a Euclidean plane,
 tessellated by regular triangles. We can distinguish several types of boundary points. Let us choose some root basis $\{a_1,a_2\}$ in the vectorial system of roots. Then there is a point $\xi\in \Csph(A)$ defined by $\Phi(\xi)=\{a_1+k,a_2+l|\,k,l\in\Z\}$. There are $6$ such points, which correspond to a choice of positive roots, {\it i.e.}, to a Weyl chamber in $A$. The sequences of (affine) chambers that converge to these points are the sequences that eventually stay in a given sector, but whose distances to each of the two walls in the boundary of this sector tend to infinity.

There is also another category of boundary points, which corresponds to sequence of residues that stay in a given sector, but stay at bounded distance of one of the two walls defining this sector. With a choice of $a_1$ and $a_2$ as before, these are points associated to set of roots of the form $\{a_1+k,a_2+l\mid k,l\in\Z,k\leqslant k_0\}\cup\{ -a_1 -k \mid k > k_0 \}$, or $\{a_1+k,a_2+l\mid k,l\in\Z ,k\leqslant k_0\}\cup\{-a_1-k_0\}$. As $k_0$ varies, we get a ``line'' of such points, and there are $6$ such lines.

When $X$ is a building of type $\tilde A_2$, as we have seen in Proposition \ref{phi}, we can always write a point in the boundary of $X$ as a point in the boundary of some apartment of $X$. Thus, the description above gives us also a description of the compactification of $X$.
\end{example}

In this example, as in the more general case of affine buildings associated to $p$-adic groups, the compactification we get was studied in \cite{GuR}.

\subsection{Sectors}

A very useful tool in the study of the combinatorial compactification is the notion of combinatorial sectors. These sectors generalize the classical notion of sectors in affine buildings: these classical sectors are also sectors in our sense. Our notion of sectors also include what G. Rousseau called ``chemin\'ees'' in \cite{Rou}.

Recall that the \emph{combinatorial convex closure} of two residues $x$ and $y$ of $X$, denoted by $\Conv(x,y)$, is the intersection of all the half-apartments that contain them.

\begin{definition}
Let $x\in \Rsph(X)$ and $\xi\in \Csph(X)$. Let $(x_n)$ be a sequence of residues converging to $\xi$. We let $$Q(x,\xi)=\bigcup_{n=0}^\infty \bigcap_{k\geqslant n}\Conv(x,x_k),$$
and call it the (combinatorial) \emph{sector} from $x$ to $\xi$.
\end{definition}

If $x$ and $y$ are in $\Rsph(X)$, we know that $x$ and $y$ are contained in an apartment. Since apartments are convex, we can deduce  that $\Conv(x,y)$ is also contained in an apartment, and hence is $W$-isometric to some subset of $\Sigma$. It is easy to deduce from this fact that $Q(x,\xi)$ is $W$-isometric to a subset of $\Sigma$, and hence is contained in an apartment by Theorem \ref{lemTits}.

It is not obvious \textit{a priori} that this $Q(x,\xi)$ does not depend on the choice of the sequence $(x_n)$. It is nevertheless the case, as the following proposition \cite[Proposition 2.27]{CL} proves:

\begin{proposition}\label{sectorroots}
 Let $\xi\in \Csph(X)$ and $x\in X$. Let $A$ be an apartment containing $Q(x,\xi)$. Then $Q(x,\xi)$ is the intersection of the roots of $A$ which are in $\Phi_A(\xi)\cap\Phi_A(x)$.
\end{proposition}

An interesting fact is that the notion of sectors behaves well with respect to the topology of $\Csph(X)$.

\begin{proposition}[{See \cite[Corollary 2.18]{CL}}]\label{induction}
  Let $\xi\in \Csph(X)$ and $x\in \Rsph(X)$. Let $(\xi_n)$ be a sequence of elements of $\Csph(X)$ which converges to $\xi$. Then $Q(x,\xi)$ is the pointwise limit of $Q(x,\xi_n)$.
\end{proposition}

The fundamental fact about sectors, which makes them useful, is that they always intersect when they have the same boundary point \cite[Proposition 2.30]{CL}:

\begin{proposition}\label{intersect}
 Let $x$ and $y$  in $\Rsph(X)$ and $\xi\in \Csph(X)$. Then $Q(x,\xi)\cap Q(y,\xi)\neq\varnothing$.
\end{proposition}

\begin{example}
 In the case of a tree, seen as a building of type $D_\infty$, the sector $Q(x,\xi)$ is the set of edges on the geodesic ray from $x$ to $\xi$. 

In the case of a regular tree of valency $r$, seen as a Coxeter complex for the group $(\Z/2\Z)^{\ast r}$, the sector $Q(x,\xi)$ is the set of vertices on the geodesic ray from $x$ to $\xi$.
\end{example}

\section{Trees in Coxeter complexes}\label{arbres} 
In this section, we explain how to construct trees in a general Coxeter complex. Furthermore, we prove that there are enough trees to fully encode the Coxeter complex. This will be a useful tool to prove amenability. The ideas that are presented here are essentially contained in \cite{Ja} and \cite{NV}.

\subsection{Construction of the trees}

Let $(W,S)$ be a Coxeter system, with $S$ finite. Let $\Sigma$ be the associated Coxeter complex. The complex $\Sigma$ has a geometric realization, called its Davis-Moussong realization. This construction is described in details for example in \cite{Dav08}. In this realization, the chambers ({\it i.e.}, vertices in $\Sigma$) are replaced by some suitable compact, piecewise Euclidean, metric space. We denote this geometric realization by $|\Sigma|$. It is a locally compact CAT(0) metric space. A wall in $\Sigma$ is geometrically realized as a convex subspace of $|\Sigma|$ which divides $|\Sigma|$ into two connected components. We also call a wall the geometric realization of a wall.

Since we know that there is an injection from $W$ to some $\mathrm{GL}_n(\C)$, Selberg's lemma \cite{Al} allows us to find a normal torsion-free finite index subgroup $W_0$ of $W$. We will now see that the orbits of walls in $|\Sigma|$ give rise to regular trees. This remark has already be done in \cite{NV} or \cite{Ja}.
The following lemma is well-known, see \cite{Mil} or \cite[Lemma 3.3]{NV}:

\begin{lemma}
Let $g\in W_0$, and let $H$ be a wall of $|\Sigma|$, of associated reflection $r$. Then either $gH=H$, or $gH\cap H=\varnothing$. In the first case, $r$ and $g$ commute.\begin{flushright}$\square$\end{flushright}
\end{lemma}

Consequently, the walls of the $W_0$-orbit of some given wall $H$ have pairwise empty intersection. They divide $|\Sigma|$ into connected components. Let us consider the graph $T_{W_0}(H)$ whose vertices are these components, and are linked by an edge if they are adjacent, so that a wall in the $W_0$-orbit of $H$ is represented by an edge of $T_{W_0}(H)$.

It is clear that $T_{W_0}(H)$ is connected. Furthermore, removing an edge to $T_{W_0}(H)$ corresponds to removing a wall in $|\Sigma|$, and it turns it into a non-connected space; as the different walls in $W_0.H$ do not intersect, $T_{W_0}(H)$ is also divided into two connected components. Thus $T_{W_0}(H)$ is a tree.% Note that  this tree need not be locally finite.

By construction, $W_0$ acts transitively on the edges of the tree. Furthermore, since $W_0$ is normal in $W$, one can define a simplicial action of $W$ on the set of trees $T_{W_0}(H)$, with $W_0$ fixed. This action is defined by $w.T_{W_0}(H)=T_{W_0}(wH)$. Indeed, for all $g\in W_0$, there exists $g'\in W_0$ such that $wgH=g'wH$.
Furthermore, since $W_0$ is of finite index in $W$, and there is a finite number of $W$-orbits of walls, the $W_0$-orbits of walls are also finite in number. Thus, there exists a finite number of walls $H_1,\dots,H_l$, such that each wall of $\Sigma$ appears as an edge in exactly one of the $T_{W_0}(H_i)$, for each $1\leqslant i\leqslant l$. Let us set $T_i=T_{W_0}(H_i)$. %The action previously defined provides an action of $W$ on $T_1\times T_2\dots\times T_l$.

\subsection{Encoding the Coxeter complex}\label{encoding}

In the preceeding section, we defined a finite number of trees $T_i$, with $1 \leqslant i\leqslant l$.  We also have seen that $W_0$ acts on each of these trees, and that $W$ acts on this set of trees by permuting them. More precisely, let $w\in W$. The image by $w$ of a wall in $\Sigma$ is another wall in $\Sigma$, and thus the image of some edge in some $T_i$ is an edge in some tree $T_{\sigma(i)}$, for some permutation $\sigma$ associated to $w$. Furthermore, if two edges are adjacent in $T_i$, then their images are again adjacent in $T_j$. So, from the action on the set of edges we get an action on the set of vertices on $T_1\cup\dots\cup T_l$. We define the action of $W$ on $T_1\times\dots\times T_l$ to be the diagonal action.

Let $r_i$ be the valency of the homogeneous tree $T_i$. Note that $r_i$ can be any integer, and can also be infinite (countable). We know that $T_i$ can be seen as the Coxeter complex of the Coxeter group $W^{r_i}\simeq (\Z/2\Z)^{*r_i}$. This allows us to speak of roots in $T_i$ and of the combinatorial compactification of $T_i$.

Using the arguments above, we can prove:

\begin{lemma}\label{lem:root}
 There is a $W$-equivariant bijection $\Psi$ between the set $\Phi(\Sigma)$ of roots of $\Sigma$ and the set $\Phi(T_1\dots\times\dots\times T_l)$ of roots in $T_1\dots\times\dots\times T_l$.
\end{lemma}
\begin{proof}
 First of all, note that any root in $T_1\times\dots\times T_l$ is just a product $\tld\alpha_i:=T_1\times\dots\times T_{i-1}\times\alpha_i\times T_{i+1}\times\dots\times T_l$, where $\alpha_i$ is a root in $T_i$. Since an element $w\in W$ acts by permutation of the trees, it sends a ``half-tree'' on a ``half-tree'', and hence we see that $W$ acts on $\Phi(T_1\times\dots\times T_l)$.

If $\alpha$ is a root in $\Sigma$, and $H$ its boundary wall, then there is a unique $1\leqslant i\leqslant l$ such that $T_{W_0}(H)=T_i$. Then $\alpha$ is a connected component of $\Sigma\setminus H$, and hence is a union of connected components of $T_{W_0}(H)=T_i$. So it defines a subset of $T_i$; furthermore, this subset is a connected component of $T_i$ deprived of the edge corresponding to $H$. Hence it defines a root $\alpha_i$ in $T_i$, and we can define $\Psi(\alpha)$ to be the root $\tld\alpha_i$ of $T_1\times\dots\times T_l$.

This defines the map $\Psi$ from $\Phi(\Sigma)$ to $\Phi(T_1\times\dots\times T_l)$. This map has an inverse. Indeed, if $\tld\alpha_i$ is a root as above, so that $\alpha_i$ is a root in $T_i$, then $\alpha_i$ is a connected component of $T_i$ deprived of an edge corresponding to a wall $H$. Then the union of all connected components corresponding to vertices in $\alpha_i$ is a connected component of $\Sigma\setminus H$, so it is a root $\alpha$ in $\Sigma$. Obviously, this map from $\Phi(T_1\times\dots\times T_l)$ to $\Phi(\Sigma)$ is the inverse of $\Psi$.

Furthermore, $\Psi$ is equivariant. Indeed, let $\alpha$ be a root in $\Sigma$ and $\tld\alpha_i=\Psi(\alpha)$, and let $H$ be the boundary wall of $\alpha$. Then $H$ corresponds to an edge $e$ in $T_i$, which is by definition the boundary wall of $\alpha_i$. Hence, if $w\in W$, then $wH$ corresponds, by definition of the action, to the edge $we$ in some $T_j$, which is also the boundary wall of the root $\tld\beta_j:=w\tld\alpha_i$. Moreover, if $H'$ is a wall which corresponds to an edge $e'$ in $T_i$ which is contained in $\alpha_i$, then $H'$ itself is contained in $\alpha$. Then we see that $we'$ is an edge contained in $\beta_j$. By definition, $\Psi(w\alpha)$ is the root of $T_1\times\dots\times T_l$ with boundary wall $T_1\times\dots\times T_{j-1}\times e \times T_{j+1}\times\dots\times T_l$ and which contains $we'$. So there is no choice: $\Psi(w\alpha)$ is equal to $\tld\beta_j=w\tld\alpha_i$, which proves the equivariance of $\Psi$.
\end{proof}

\begin{lemma}\label{lem:embed1}
 Let $W$ be a Coxeter group and $\Sigma$ its Coxeter complex. Then $\Rsph(\Sigma)$ is embedded $W$-equivariantly in $\Rsph(T_1\times\dots\times T_l)=\Rsph(T_1)\times\dots\times\Rsph(T_l)$. 
\end{lemma}

\begin{proof}

Let $R$ be a spherical residue in $\Sigma$. The intersection of the geometric realization of chambers in $R$ is not empty. We see it as the geometric realization of $R$ and we denote it by $|R|$. Note that if $R$ is contained in some root $\alpha$, then $|R|$ is contained in its geometric realization $|\alpha|$; furthermore, $|R|$ is equal to the intersection of all the roots which contain it.

Let $H$ be a wall in $\Sigma$. Assume first that $R$ does not intersect any wall in $W_0.H$. Then $|R|$, being connected, is contained in a unique connected component of $|\Sigma|\setminus W_0.|H|$. So we can associate to $R$ a unique vertex $\psi_H(R)$ in $T_{W_0}(H)$. Now, if $|R|$ intersects some wall in $W_0.H$, then it is contained in this wall. Since all the geometric realization of walls in $W_0.H$ do not intersect pairwise, this wall is unique. Hence we can associate to $R$ the edge of $T_{W_0}(H)$ which corresponds to the wall containing $|R|$.

Thus, we get a map $$\psi_H:\Rsph(\Sigma)\to\Rsph(T_{W_0}(H)).$$
If $T_i=T_{W_0}(H_i)$, the product $\psi:\psi_{H_1}\times\dots\times\psi_{H_l}$ is a map from $\Rsph(\Sigma)$ to $\Rsph(T_1)\times\dots\times\Rsph(T_l)$. 

What is left to prove is that this map is injective and $W$-equivariant. Let $\Phi(R)$ be the set of roots which contain $R$. Note that, if $R\neq S$, then $\Phi(R)\neq\Phi(S)$. Then we see that $\Phi(\psi(R))=\Psi(\Phi(R))$. The injectivity and equivariance of $\Psi$ then proves the equivariance of $\psi$.
%
%Let $R$ and $S$ be two different residues. Then there exists a root $\alpha$, with boundary wall $H$, such that $R$ is contained in $\alpha$, but $S$ is not. Then $|R|$ and $|S|$ will be in different connected components of $|\Sigma|\setminus W_0.|H|$. Since $T_{W_0}(H)$ is one of the trees $T_i$, it proves that $\psi$ is injective.
%
%Now, we have to prove that $\psi$ is equivariant. Let $R\in\Rsph(\Sigma)$ and $w\in W$. Let $\psi(w.R)=(S_1,\dots,S_l)$. Note that $\psi(w.R)$ is uniquely determined by the roots which contain $S_i$ in the tree $T_i$. These set of all these roots, when $i$ varies, corresponds to the set of roots in $\Sigma$ containing $w.R$, {\it i.e.} to $\Phi(w.R)=w.\Phi(R)$. By definition of the action, $w\psi(R)$ is the unique point in $T_1\times\dots\times T_l$ which is contained in every root in each $T_i$ wich corresponds to a root in $w.\Phi(R)$. So, we have $\psi(w.R)=w.\psi(R)$, and $\psi$ is equivariant. 
\end{proof}

Using the same kind of arguments, we prove that this embedding extends to an embedding of the combinatorial compactification:

\begin{lemma}\label{lem:embed2}
 There exists an injective, $W$-equivariant map $\phi:\Csph(\Sigma)\to \Csph(T_1)\times\dots\times \Csph(T_l)$, such that $\phi|_{\Rsph(\Sigma)}=\psi$. Furthermore, the image by $\phi$ of a sector in $\Sigma$ is embedded into a product of half-lines and segments.
\end{lemma}

\begin{proof}

Let $\xi\in\Csph(\Sigma)$. We define $\phi(\xi)$ by the set of roots which contain it: we put $\Phi(\phi(\xi))=\Psi(\Phi(\xi))$. Note that the restriction $\phi|_{\Rsph(\Sigma)}$ is equal to $\psi$.

We have to prove that this $\phi(\xi)$ is indeed a point in $\Csph(T_1\times\dots\times T_l)$; in other words, that there exists a sequence of spherical residues which converge to it. Let $(R_n)$ be a sequence of spherical residues converging to $\xi$. This means that, for every root $\alpha\in\Phi(\xi)$ and for $n$ large enough, the residue $R_n$ is contained in $\alpha$. In other words, for $n$ large enough, we have $\alpha\in\Phi(R_n)$, and we deduce that $\Psi(\alpha)\in\Psi(\Phi(R_n))$. This means that $\Psi(\alpha)\in\Phi(\psi(R_n))$, so that $(\psi(R_n))$ converges to $\phi(\xi)$.

Once again, the equivariance and injectivity follow from the equivariance and injectivity of $\xi$.

Let $x\in\Rsph(\Sigma)$ and $\xi\in \Csph(\Sigma)$. Let $\psi(x)=(x_1,\dots,x_l)$ and $\psi(\xi)=(\xi_1,\dots,\xi_l)$. We know that $Q(x,\xi)$ is the intersection of the roots in $\Phi(\xi)\cap\Phi(x)$. In particular, the projection of $\psi((Q(x,\xi))$ on $T_1$ is the intersection of the roots in $T_1$ containing both $x_1$ and $\xi_1$. Hence, it is a sector in $T_1$. So it is either a half-line or a segment. Since it is the same for every tree, we get that $\psi(Q(x,\xi))$ is contained in a product of half-lines.
\end{proof}

The proof of the amenability relies on these facts. They allow us to work first with a product on trees instead of arbitrary Coxeter complexes. Since we know that the action on the boundary of a tree is amenable, it makes things much easier. Furthermore, the simple structure of sectors in this product of trees makes them easy to handle.

Let $y$ be a chamber in a sector $Q(x,\xi)$, where $x\in\Rsph(\Sigma)$ and $\xi\in \Csph(\Sigma)$. Let $\psi(x)=(x_1,\dots,x_l)$ and $\psi(y)=(y_1,\dots,y_l)$. Then $\psi(y)$ is in $\psi(Q(x,\xi))$, which is a product of half-lines and segments. So it is uniquely determined by an $l$-tuple of integers $d(x_1,y_1),\dots,d(x_l,y_l)$.

\begin{definition}\label{defposition}
 We call this $l$-tuple of integers  the \emph{position of $y$ towards $(x,\xi)$}.
\end{definition}

\section{Amenability}\label{amenability}

We recall some basic facts about amenability. Our reference is \cite{AR}.

Amenability of a group is a well-known notion, see for example \cite{Pat}. There are many different definitions of amenability, which can fortunately be proved to be equivalent. One of them is the following: a group $G$ is amenable if every continuous action on a compact space has an invariant probability measure. It implies more generally that $G$ fixes a point in every compact, convex subset of the unit ball of a dual Banach space with the weak-* topology; $G$ acting by affine maps. Another definition of amenability is that $G$ has an invariant mean.

This definition can be generalized to actions of groups (in fact, to groupoids). This generalization was first made by Zimmer \cite{Zim}. His first definition was a generalization of the fixed point property: this definition is roughly the existence of a fixed point on a bundle of affine spaces over $S$. This definition can be found for example in \cite[4.3]{Zimmer}. In the case of discrete groups, he also proves his definition to be equivalent to Definition \ref{def:expectation}, which is a generalization of the existence of an invariant mean in some sense. This result was generalized in \cite{AEG} to locally compact groups. In this setting, the notion of mean should be generalized by the notion of a conditional expectation:

\begin{definition}\label{def:expectation} 

Let $X$ be a space equipped with a measure $\mu$. Let $G$ be a locally compact group acting on $X$ in such a way that $\mu$ is quasi-invariant.
 A \emph{conditional expectation} $m:L^\infty(G\times X)\to L^\infty(X)$ is a linear, continuous map such that
\begin{itemize}
 \item $m$ is of norm one;
 \item $m(\mathbf 1_{G\times X})=\mathbf 1_X$;
 \item for any $f\in L^\infty(G\times X)$ and any measurable subset $A\subset X$ we have the equality $m(f (\mathbf 1_{G\times A}))=m(f) \mathbf 1_A$.
\end{itemize}
The action of $G$ on $X$ is called \emph{amenable in the sense of Zimmer's} if there is a $G$-equivariant conditional expectation $m:L^\infty(G\times X)\to L^\infty(X)$ (where $G\times X$ is endowed with the diagonal action $g.(h,x)=(gh,g.x)$)
\end{definition}

This definition does not much take into account the topology of $G$: one only needs the Haar measure on $G$ in order to state it. There are also some notions of amenability which take more topology into account. Unfortunately, these notions are not equivalent to the amenability in the sense of Zimmer. A naive idea would be to require the existence of a continuous system of probability measures $m:X\to \PR(G)$ which is $G$-equivariant. However, this notion would be too strong: in the case of a group acting on a point, this would be equivalent to having a finite Haar measure, which means that the group would be compact. So, the condition has to be relaxed. It can be done by requiring the existence of continuous maps $m:X\to \PR(G)$ which are \emph{asympotically} equivariant:

\begin{definition}\label{critere}
Let $X$ be a locally compact space endowed with a continuous action of a locally compact group $G$. The action of $G$ on $X$ is said to be \emph{topologically amenable} if there exists a sequence of continous maps $m_n:X\to\PR(G)$ such that $\Vert m_n(g.x)-g.m_n(x)\Vert$ tends to $0$ uniformly on every compact subset of $G\times X$.
\end{definition}

This is the definition we will use. Here the norm we use on $\PR(G)$ is the total variation, which can be seen as the dual norm of the norm on continuous compactly supported functions. The continuity of $m_n$ is to be understood with respect to the weak-* topology on $\PR(G)$. As we said before, topological amenability is not the same as amenability in the sense of Zimmer. In fact, it is a stronger notion:

\begin{proposition}[{\cite[Proposition 3.3.5]{AR}}]
Let $X$ be locally compact space with a continuous action of the locally compact group $G$. If the action of $G$ on $X$ is topologically amenable, then for every quasi-invariant Borel measure $\mu$ on $X$, the action of $G$ on $X$ is amenable in the sense of Zimmer.\begin{flushright}$\square$\end{flushright}
\end{proposition}

Amenability of groups is well-behaved under taking subgroups and quotients. The corresponding properties, suitably defined, are still true for amenable actions. In particular, \cite[Proposition 5.1.1]{AR} proves:

\begin{proposition}\label{subgroup}
 Let $G$ be a locally compact group acting on a space $X$. Let $H$ be a closed subgroup of $G$. Assume the action of $G$ on $X$ is topologically amenable. Then the action of $H$ on $X$ is topologically amenable.\begin{flushright}$\square$\end{flushright}
\end{proposition}

In the case of a group $G$ acting on a building $X$, in order to define this sequence $(m_n)$, we will, as in \cite{Ka}, define some $G$-equivariant Borel maps $\mu_n:X\times\Csph(X)\to\PR(X)$. Then, if $o\in X$, we see that $\mu'_n=\mu_n(o,\cdot)$ is a map from $\Csph(X)$ to $\PR(X)$ which is Borel and such that $\Vert g.\mu'_n(\xi)-\mu'_n(g\xi)\Vert$ tends to $0$ uniformly on every compact subset of $G\times\Csph(X)$. Then we use the following proposition, which is a particular (easy) case of \cite[Proposition 11]{Oz}. We will use this proposition with $D$ equal to the set of chambers of $X$ and $Y$ equal to $\Csph(X)$.

\begin{proposition}\label{prop:critere}
 Let $G$ be a locally compact, second countable group, acting continuously and properly on a discrete, countable set $D$ and acting continuously on a metric, locally compact space $Y$.

Assume that, for every compact subset $Q\subset G$ and every $\eps >0$, there exists a continuous map $\zeta:K\to \PR(D)$ such that:
$$\sup_{q\in Q}\sup_{x\in K}\Vert \zeta(qx)-q\zeta(x)\Vert\leqslant\eps.$$

Then the action of $G$ on $Y$ is amenable.
\end{proposition}
Here $\PR(D)$ is endowed with the weak-* topology, which, in this case, coincides in fact with the norm topology.
\begin{proof}
Let $Q$ be a compact subset of $G$, $\eps>0$, and $\zeta:Y\to \PR(D)$ the associated map.

Let $V$ be a fundamental domain for the $G$-action  on $D$, with  associated projection $v:D\to V$.
Let $a\in D$, $g\in G$ and $b=ga$. Let $G_a$ be the stabilizer of $a$ and let $H_a:=\Haar(G_a)$ denote its Haar measure, viewed as a measure on $G$ with support in $G_a$. Then $g.H_a$ is a probability measure on $G$ which only depends on $a$ and $b:=ga$, we denote it by $H_a^b$.

We define $\mu:Y\to\PR(G)$ by the following formula:
$$\mu(x) = \sum_{a\in D} \zeta(x)(a) H^a_{v(a)}.$$

Then $\mu$ is continuous, even for the norm topology on $\PR(G)$, since $\Vert\mu_{y}-\mu_x\Vert\leqslant\Vert\zeta_{y}-\zeta_x\Vert$
and $\zeta$ is continuous with respect to the norm topology on $\PR(D)$.

Furthermore we have, for $q\in Q$ and $x\in K$,
$$q.\mu(x)=\sum_{a\in D}\zeta(x)(a) H^{qa}_{v(a)},$$
so that
\begin{eqnarray*} 
\Vert q\mu(x)-\mu(qx)\Vert&=& \Vert\sum_{a\in D} (\zeta(x)(a)-\zeta(qx)(a)) H^{qa}_{v(a)}\Vert\\
&\leqslant& \sum_{a\in D} \vert \zeta(x)(a) - \zeta(qx)(a)\vert\\
&\leqslant& \eps,
\end{eqnarray*}

which proves the proposition.
\end{proof}
\begin{corollary}\label{cor:moyennable}
 Let $X$ be a building and $G$ be a locally compact subgroup of $\Aut(X)$ which acts properly on $\Rsph(X)$. Assume that there exists a sequence of $G$-equivariant continuous maps $\mu_n:\Rsph(X)\times\Csph(X)\to\PR(\Rsph(X))$  such that $\Vert\mu_n(x,\xi)-\mu_n(x',\xi)\Vert$ tends to $0$ uniformly on $\Csph(X)$ for every $x,x'\in \Rsph(X)$.

Then the action of $G$ on $\Csph(X)$ is amenable.
\end{corollary}

\begin{proof}
Fix a spherical residue $x$, and let $m_n(\xi)=\mu_n(x,\xi)$. Then $m_n$ is a continuous map from $\Csph(X)$ to $\PR(\Rsph(X))$, and furthermore we have $$\Vert m_n(g\xi) - g.m_n(\xi)\Vert = \Vert\mu_n(x,g\xi)-\mu_n(gx,g\xi)\Vert.$$
If $g$ is in a compact subset of $G$, then $gC$ only runs through a finite number of chambers, so we see that $\lim_n\Vert m_n(g\xi)-g.m_n(\xi)\Vert=0$ uniformly on every compact of $G\times\Csph(X)$. In view of Proposition \ref{prop:critere}, this proves that the action of $G$ on $\Csph(X)$ is amenable.

\end{proof}

We proceed in two steps to define the maps $\mu_n$: first, we define them in an apartment (Section \ref{section:amenCox}, then we glue them to get a measure on the whole building (Section \ref{section:amenImm}.

\section{Amenability for Coxeter complexes}\label{section:amenCox}

Let $W$ be a Coxeter group. We want to prove that the action of $W$ on the combinatorial compactification $\Csph(\Sigma)$ of its Coxeter complex $\Sigma$ is amenable. In view of Definition \ref{critere}, we have to define a sequence of continuous functions $m_n:\Csph(\Sigma)\to\PR(W)$ which are asymptotically equivariant. As in Corollary \ref{cor:moyennable}, we construct in fact a sequence of $W$-equivariant map $\mu_n:\Rsph(\Sigma)\times\Csph(\Sigma)\to\PR(\Rsph(\Sigma))$ such that 
$$\lim_{n\to +\infty}\Vert\mu_n(x,\xi)-\mu_n(x',\xi)\Vert=0,$$
 uniformly on $\xi\in\Csph(\Sigma)$.

The measures we define here will be later glued together in order to get measures on a given thick building. To be able to do this gluing, we need in fact a more precise information on the support of these measures.

\begin{proposition}\label{appt}
There exists a sequence of continuous $W$-equivariant maps $$\lambda_n:\Sigma\times\Csph(\Sigma)\to\PR(\Sigma),$$
such that, for any $n\in \N$:
\begin{enumerate}[(i)]
 \item  For any $x\in\Rsph(\Sigma)$ and $\xi\in \Csph(\Sigma)$, we have: $\Supp(\lambda_n(x,\xi))\subset Q(x,\xi)$
 \item For any $x,x'\in\Rsph(\Sigma)$, we have $\lim_n \Vert\lambda_n(x,\xi)-\lambda_n(x',\xi)\Vert=0$, uniformly in $\xi$.
\end{enumerate}
\end{proposition}

When $Z$ is a finite subset of some discrete set $Y$, let us denote by $m_Z$ the measure defined on $Y$ by $m_Z(A)=\frac{|A\cap Z|}{|Z|}$. The proof we make is inspired from \cite[Theorem 1.33]{Ka}.  More precisely, we use the following lemma, which can be proved by a simple calculation:
 
 \begin{lemma}[See {\cite[Lemma 1.35]{Ka}}]\label{Ka}
 Let $\{Z_k\}_{k\geqslant 1}$ et $\{Z'_k\}_{k\geqslant 1}$ be increasing sequences of finite subsets of some discrete set $Y$. Assume there is an integer $\tau>0$ such that, for any $k\geqslant 1$:
 \begin{equation}
 Z_k\subset Z'_{k+\tau},\hspace{1cm}Z'_k\subset Z_{k+\tau}. \label{sandwich}
 \end{equation}
 Let us define the measures:
 $$\lambda_n=\frac 1 n \sum_{k=1}^n m_{Z_k},\hspace{1cm}\lambda'_n=\frac 1 n \sum_{k=1}^n m_{Z'_k}.$$
Then $$\Vert \lambda_n-\lambda'_n\Vert\leqslant \frac{2\tau} n + \frac{4(n-\tau)} n\left[1-\left(\frac {|Z_1|}{|Z_{n+\tau}|}\right)^{\frac{2\tau}{n-\tau}}\right],$$
for all $n>\tau$. 
 \end{lemma}

In particular, if $\lim_n |Z_n|^{1/n}=0$, then $\lim_n\Vert \lambda_n-\lambda'_n\Vert=0$.

We divide the proof of Proposition \ref{appt} into three steps. We first deal with the case of a tree, for which the amenability is a particular case of \cite{Ka} for example (but we need to prove a little more to control the support). Then we use this particular case to deal with products of trees. Finally, we use the embedding of a general Coxeter complex in a product of trees described in Section \ref{encoding} in order to prove Proposition \ref{appt}.

\subsection{The case of  trees}

If $W=(\Z/2Z)^{\ast r}$, the Coxeter complex of $W$ is a regular tree $T$ of valency $r$. In this case, $\Rsph(T)$ is equal to the set of vertices and edges of $T$. The combinatorial boundary $\Csph(T)$ is equal to the union of $\Rsph(T)$ with the usual boundary $\partial_\infty T$.

\begin{lemma}
 Let $T$ be a regular tree of valency $r\in \N\cup\{\infty\}$. Let $W=(\Z/2\Z)^{\ast r}$. There exist $W$-equivariant maps $$\lambda_n:\Rsph(T)\times \Csph(T)\to\PR(\Rsph(T)),$$
such that, for any $n\in \N$, $x\in\Rsph(T)$ and $\xi\in \Csph(T)$:
\begin{enumerate}[(i)]
 \item  For any $n\in \N$, $x\in\Rsph(T)$ and $\xi\in \Csph(T)$ we have: $\Supp(\lambda_n(x,\xi))\subset Q(x,\xi)$
 \item For any $x,x'\in \Rsph(T)$, we have $\lim_{n\to \infty}\Vert\lambda_n(x,\xi)-\lambda_n(x',\xi)\Vert=0$ uniformly in $\xi$.
\end{enumerate}
\end{lemma}

\begin{proof}
In the following we denote by $d$ the root-distance between residues in $\Rsph(T)$. By definition, $d$ is the usual distance on the set of vertices and on the set of edges, and a vertex is at distance $1/2$ of its adjacent edges. Let $x\in \Rsph(T)$ and $\xi$ a point in $T\cup\Csph(T)$. Note that in this case, $Q(x,\xi)$ is either a segment or a half-line; both being seen as a union of vertices and edges.

If $\xi\in \Csph(T)\setminus\Rsph(T)$ and $k\leqslant n$, we define the following set of vertices: $$Z(x,\xi,n,k)=\{z\in Q(x,\xi)\,|\, n-k \leqslant d(z,x)\leqslant n+k\}.$$
If $\xi\in T$, we define in the same way, for $n-d(x,\xi)\leqslant k\leqslant n$,$$Z(x,\xi,n,k)=\{z\in [x,\xi]\,|\, n-k \leqslant d(z,x)\leqslant n+k\}\cup\{\xi\},$$
and if $k\leqslant n- d(x,\xi)$, we put $Z(x,\xi,n,k)=\xi$.

We define then $\lambda_n(x,\xi)$ as in Lemma \ref{Ka} by a Ces\`aro average $$\lambda_n(x,\xi)=\frac 1 n \sum_{k=1}^nm_{Z(x,\xi,n,k)}.$$

 Let $x$ and $x'$ be two spherical residues in $T$, $\xi\in \Csph(T)$ and $\tau\geqslant d(x,x')$. Then the sequences of sets $\{Z(x,\xi,n,k)\}_{k=1}^n$ and $\{Z(x',\xi,n,k)\}_{k=1}^n$ are ``$\tau$-sandwiched'' in the sense of (\ref{sandwich}). Indeed, let $y\in Z(x,\xi,n,k)$. Assume first that $\xi\not\in \Rsph(T)$ or that $k\geqslant n-d(x,\xi)$. Then $d(x,y)\leqslant d(x,x')+d(x',y)\leqslant \tau+n+k$, and $d(x',y)\geqslant d(x,y)-d(x,x')\geqslant n-k-\tau$. So $y\in Z(x',\xi,n,k+\tau)$. If $k\leqslant n-d(x,\xi)$, then we have $y=\xi$. But then, $d(x',\xi)\geqslant n-k-\tau$, so we have also $y\in Z(x',\xi,n,k+\tau)$. Hence, $Z(x,\xi,n,k)\subset Z(x',\xi,n,k+\tau)$. By symmetry, we have also $Z(x',\xi,n,k)\subset Z(x,\xi,n,k+\tau)$.

 Thus, we can apply Lemma \ref{Ka} to the sequences $\{Z(x,\xi,n,k)\}_{k=1}^n$ and 
$\{Z(x,\xi,n,k)\}_{k=1}^n$. As $|Z(x,\xi,n,k)|\leqslant 4k+1$, we see that $\Vert \lambda_n(x,\xi)-\lambda_n(x',\xi)\Vert$ tends to $0$. Furthermore, the convergence is uniform on $\xi\in\Csph(T)$.

Condition (i) is clear since $Z(x,\xi,n,k)\subset Q(x,\xi)$. Moreover, to check continuity of $\lambda_n$, it is enough to check continuity with respect to the second argument, since $\Rsph(T)$ is a discrete set. To do so, we see that if $\xi_m\in \Csph(T)$ is such that $\lim_m \xi_m=\xi$, then  $Q(x,\xi_m)$ converges to $Q(x,\xi)$, by Proposition \ref{induction}.
 Furthermore, if $\xi_m\in \Rsph(T)$ and if $n\in \N$, $x\in\Rsph(T)$ are fixed, then for  $m$ large enough, we see that $k\geqslant n-d(x,\xi)$ for any $k$. Hence, for $m$ large enough, we have $Z(x,\xi_m,n,k)=Z(x,\xi,n,k)$ for any $k$. This proves that $\lambda_n(x,\xi_m)=\lambda_n(x,\xi)$ for $m$ large enough, thus proving the continuity of $\lambda_n$.

Furthermore, $\lambda_n$ is $\Aut(T)$-equivariant: indeed, if $g\in \Aut(T)$, we have

\begin{eqnarray*}
g.\lambda_n(x,\xi)&=&\frac 1 n \sum_{k=1}^ng.m_{Z(x,\xi,n,k)}\\
&=&\frac 1 n \sum_{k=1}^nm_{Z(gx,g\xi,n,k)}\\
&=&\lambda_n(gx,g\xi).
\end{eqnarray*}

\end{proof}

\subsection{Products of trees}

We extend now this construction to the case when $\Sigma$ is a product of trees $T_1\times\dots\times T_l$. If $T_i$ is of valency $r_i$, such a $\Sigma$ is the Coxeter complex of $W= (\Z/2\Z)^{*r_1}\times\dots\times (\Z/2\Z)^{*r_l}$.

\begin{lemma}
 If $W=(\Z/2\Z)^{*r_1}\times\dots\times (\Z/2\Z)^{*r_l}$, so that $\Sigma$ is a product of regular trees, then the statement of Proposition \ref{appt} holds.
\end{lemma}

\begin{proof}
Let $x=(x_1,\dots,x_l)\in \Rsph(\Sigma)$ and $\xi=(\xi_1,\dots,\xi_l)\in\Csph(\Sigma)$. The above argument proves that there exist some maps $\lambda_n^i:\Rsph(T_i)\times\Csph(T_i)\to \PR(\Rsph(T_i))$ which satisfy the conditions of Proposition \ref{appt}. We can now define $\lambda_n(x,\xi)=\lambda_n^1(x_1,\xi_1)\otimes\dots\otimes\lambda_n^l(x_i,\xi_i)$.

As a product of continuous functions, $\lambda_n$ is continuous. Since $Q(x,\xi)=Q(x_1,\xi_1)\times\dots\times Q(x_l,\xi_l)$ and $\Supp(\lambda_n^i(x_i,\xi_i))\subset Q(x_i,\xi_i)$, we have $\Supp(\lambda_n(x,\xi))\subset Q(x,\xi)$. Furthermore, if $x'=(x'_1,\dots,x'_l)\in\Sigma$, we see that $\Vert\lambda_n(x,\xi)-\lambda_n(x',\xi)\Vert=\prod_{i=1}^l\Vert\lambda_n^i(x_i,\xi_i)-\lambda_n^i(x'_i,\xi_i)\Vert$,
  so  that 
$$\lim_{n\to +\infty} \Vert\lambda_n(x,\xi)-\lambda_n(x',\xi)\Vert=0.$$
 Since the convergence is uniform on each $\Csph(T_i)$, it is also uniform on $\Csph(\Sigma)$. Furthermore, since each $\lambda_n^i$ is $\Aut(T_i)$-equivariant, we see that $\lambda_n$ is $\Aut(T_1)\times\dots\times\Aut(T_l)$-equivariant. As the Coxeter group $W$ is included in $\Aut(T_1)\times\dots\times\Aut(T_l)$, this proves the lemma.
\end{proof}

\subsection{General Coxeter complexes}
In this section, we prove Proposition \ref{appt} for a general Coxeter complex $\Sigma$, associated with a Coxeter group $W$.

 We have seen in Section \ref{arbres} that $W$ acts on a product of trees $\widetilde\Sigma:=T_1\times\dots\times T_l$. Moreover, we know that there exists a finite index subgroup $W_0$ of $W$ that stabilizes each of the trees $T_i$. The argument above proves that there exists a continuous map $W_0$-equivariant map $\widetilde\lambda_n: \Rsph(\Sigma)\times\Csph(\Sigma)\to\PR(\Rsph(\widetilde\Sigma))$ for which  condition (ii) of Proposition \ref{appt} is satisfied, and such that $\Supp\widetilde\lambda_n(x,\xi)\subset \widetilde Q(x,\xi)$, where $\widetilde Q(x,\xi)$ is the combinatorial sector associated to $(x,\xi)$ in $\widetilde \Sigma$.

Since we want a measure on $\Rsph(\Sigma)$ and not $\Rsph(\tld \Sigma)$, we have to transform this measure into a measure whose support is really included in $Q(x,\xi)\subset \Rsph(\Sigma)$.

Let $Y$ be a convex subset of $\Rsph(\Sigma)$ -- convex meaning that $Y$ is the intersection of the roots that contain it. By Lemma \ref{lem:root}, to these roots are associated some roots of $\tld\Sigma$ by the application $\Psi$, and we define $\tld Y$ to be the intersection of the roots $\Psi(\alpha)$, for all the roots $\alpha$ containing $Y$. By construction, this set is a convex subset of $\Rsph(\tld \Sigma)$, which contains the image of $Y$ in $\Rsph(\tld \Sigma)$. Note also that $Y$, being a connected subset of $\Rsph(\Sigma)$ in the graph-theoretical sense, is also connected when viewed as a subset of $\Rsph(\tld\Sigma)$. Because of Proposition \ref{sectorroots}, the set $Q(x,\xi)$ is convex, and its associated set $\tld Q(x,\xi)$ is indeed the sector associated to $(x,\xi)$ in $\tld\Sigma$.

\begin{lemma}
Let $Y$ be a convex subset of $\Rsph(\Sigma)$. Assume that $Y$ is contained in some $Q(x,\xi)$. Then there exists an operator $S=S_Y\,:\, \ell^1(\tld Y)\to \ell^1(Y),$ such that:
\begin{enumerate}[(i)]
 \item The operator $S$ is continuous of norm $1$.
 \item If $Z\subset Y$ is a convex subset of $Y$, then $S_Z=S_Y|_{\ell^1(\tld Z)}$.
 \item If $f$ is a positive function of norm $1$, then $Sf$ is also a positive function of norm $1$.
 \item If $g\in W_0$, then $S_{gY}=gS_Y$.
\end{enumerate}
\end{lemma}

\begin{proof}

The construction of $S$ is by induction on the number $l$ of trees. Assume first that $l=1$. Then, with the notations of Lemma \ref{lem:embed1} and Lemma \ref{lem:embed2}, we see that $\tld Q(x,\xi)=Q(\psi(x),\phi( \xi))$ is a half-line or a segment. Since $\tld Y$ is a connected, non-empty subset of $\tld Q(x,\xi)$, we can identify it (as a graph) to an integer interval of the form $[0,N)$ where $N\in \N\cup\{\infty\}$. By the above remarks, we see that $Y$ is then identified to a connected subgraph of $[0,N)$, and therefore is of the form $[n_1,n_2)$ where $n_1\geqslant 0$ and $n_2\leqslant N$.

If $f\in\ell^1(\tld Y)$, then let $Sf(n_1)=\sum_{k\leqslant n_1} f(k)$, $Sf(k)=f(k)$ if $k\in [n_1+1,n_2-1)$, and $Sf(n_2)=\sum_{n_2\leqslant k< N}f(k)$ if $n_2$ is finite.

Note that $S$ is independent of the choice of one of the two choices of ordering which may be possible on $\tld Y$, and depends only on the graph structure. Therefore, if $g\in \Aut(T_1)$, then we have  $gS_Y=S_{gY}$. The other properties of $S$ are quite obvious.

Now, if $l$ is general, then we know that $\tld Q(x, \xi)$ is a product of segments and half-lines. Since $\tld Y$ is a convex subset of $\tld Q(x, \xi)$, it is also isomorphic, as a graph, to a product of segments and half-lines. Hence it is isomorphic, as above, to a product $[0,N^1)\times [0,N^2)\times\dots\times [0,N^l)$. Let $p_i$, for $1\leqslant i\leqslant l$, be the projection of $\Rsph(\tld\Sigma)$ on the graph $\Rsph(T_i)$.

The projection $p_1(Y)$ is a non-empty connected subgraph of $[0,N^1)$, so it is of the form $[n^1_1,n^1_2)$. We first define an operator $R=R_Y\,:\, \ell^1(\tld Y)\to\ell^1([n^1_1,n^1_2)\times [0,N^2)\times\dots\times [0,N^l))$ by the following formulas, where $z\in [0,N^2)\times\dots\times [0,N_l)$:
$$Rf(n_1^1,z)=\sum_{k\leqslant n_1^1} f(k,z),$$
$$Rf(n_2^1,z)=\sum_{n_2^1\leqslant k< N^1} f(k,z)$$
 if $n_2$ is finite, and finally $Rf(k,z)=f(k,z)$ if $n_1^1<k<n_2^1-1$.

Once again, it is easy to see that $R$ satisfies the four points of the proposition.

Now, if $k\in [n_1^1,n_2^1)$, then by definition $Y\cap p_1^{-1}(k)\neq\varnothing$, and this set is connected, as it is the image of an intersection of roots in $\Sigma$. By induction we have some operators $S_{Y,k}:\ell^1(\{k\}\times [0,N^2)\times\dots\times [0,N_l))\to\ell^1(p^{-1}(k)\cap Y)$. Then we define: $S_Yf(k,z)=S_{Y,k}(R(f)(k,\cdot))(z)$.

If $Z$ is a convex subset of $Y$, then by induction we see that if $p^{-1}(k)\cap Z\neq\varnothing$, then $S_{Z,k}$ is equal to $S(Y,k)$ restricted to $\ell^1(p^{-1}(k)\cap \tld Z)$. So $S_Z$ is equal to the restriction of $S_Y$ to $\ell^1(\tld Z)$.

Let $g\in W_0$. Then $g$ stabilizes each tree, so that $g=g_1\times\dots\times g_l \in \Aut(T_1)\times\cdots\times \Aut(T_l)$. Then it is easy to see that $gR_y=R_{gY}$, and that, if $x_1$ is the point of $p_1(Y)\subset T_1$ indexed by $k$, then $gS_{Y,k}=gS_{Y\cap p^{-1}(x_1)}=S_{gY\cap p^{-1}(g_1x_1)}$ by induction, so that we have $gS_Y=S_{gY}$.
\end{proof}

\begin{proof}[Proof of Proposition \ref{appt}]For $(x,\xi)\in\Rsph(\Sigma)\times\Csph(\Sigma)$, let $\lambda_n(x,\xi)=S_{Q(x,\xi)}(\tld\lambda_n(x,\xi))$. By construction, $\lambda_n$ is $W_0$-equivariant, and $\Supp(\lambda_n(x,\xi))\subset Q(x,\xi)$.

It remains to check that 
$$\lim_{n\to\infty}\Vert\lambda_n(x,\xi) - \lambda_n(y,\xi)\Vert=0,$$
 uniformly in $\xi$. We know it is already the case for $\tld \lambda_n$. We also know by \cite[Proposition 2.30]{CL} that $Q(x,\xi)\cap Q(y,\xi)$ is a nonempty convex subset of $\Sigma$. So, we have that 
$$S_{Q(x,\xi)\cap Q(y,\xi)}=S_{Q(x,\xi)}|_{\ell^1(Q(x,\xi)\cap Q(y,\xi))}.$$ 

Now, \begin{multline}\label{eqnlambda}
\Vert\lambda_n(x,\xi) - \lambda_n(y,\xi)\Vert= \lambda_n(x,\xi)(Q(x,\xi)\setminus Q(y,\xi) )\\+ \lambda_n(y,\xi)(Q(y,\xi)\setminus Q(x,\xi))+\vert\lambda_n(x,\xi)-\lambda_n(y,\xi)\vert(Q(x,\xi)\cap Q(y,\xi)).
\end{multline}

The third term can be easily estimated: 
$$\vert\lambda_n(x,\xi)-\lambda_n(y,\xi)\vert(Q(x,\xi)\cap Q(y,\xi))= S_{Q(x,\xi)\cap Q(y,\xi)}(\vert\tld\lambda_n(x,\xi)-\tld\lambda_n(y,\xi)\vert)(Q(x,\xi)\cap Q(y,\xi)).$$ 
Since $S$ is continuous of norm $1$, we see that 
$$\vert\lambda_n(x,\xi)-\lambda_n(y,\xi)\vert(Q(x,\xi)\cap Q(y,\xi))\leqslant \vert\tld\lambda_n(x,\xi)-\tld\lambda_n(y,\xi)\vert(Q(x,\xi)\cap Q(y,\xi)),$$
 which converges to $0$ uniformly in $\xi$.

Now we estimate $\lambda_n(x,\xi)(Q(x,\xi)\setminus Q(y,\xi))=S_{Q(x,\xi)}(\tld\lambda_n(x,\xi))(Q(x,\xi)\setminus Q(y,\xi))$. Using the fact that

$$S_{Q(x,\xi)}\tld\lambda_n(x,\xi)=S_{Q(x,\xi)\cap Q(y,\xi)}(\tld\lambda_n(x,\xi)|_{Q(x,\xi)\cap Q(y,\xi)})+S_{Q(x,\xi)}(\tld\lambda_n(x,\xi)|_{Q(x,\xi)\setminus Q(y,\xi)}),$$
we see that $S_{Q(x,\xi)}(\tld\lambda_n(x,\xi)(Q(x,\xi)\setminus Q(y,\xi))\leqslant \tld\lambda_n(x,\xi)( Q(x,\xi)\setminus Q(y,\xi))$. 
Using the equation (\ref{eqnlambda}) with $\tld\lambda$ instead of $\lambda$, we have
$$\lim_{n\to+\infty}\tld\lambda_n(x,\xi)(Q(x,\xi)\setminus Q(y,\xi))=0,$$
so that $$\lim_{n\to+\infty}\lambda_n(x,\xi)(Q(x,\xi)\setminus Q(y,\xi))=0.$$

By symmetry, we have that $\lim_n \lambda_n(y,\xi)(Q(y,\xi)\setminus Q(x,\xi))=0$. Putting the pieces together, we have  $\lim_{n\to\infty}\Vert\lambda_n(x,\xi) - \lambda_n(y,\xi)\Vert=0$  uniformly in $\xi$.\\

We have also to check the continuity of $\lambda_n$. Let $\xi,\xi'\in\Csph(\Sigma)$. Since we already know that $\tld\lambda_n$ is continuous, we have to prove that if $\Vert \tld\lambda_n(x,\xi)-\tld\lambda_n(x,\xi')\Vert\leqslant \varepsilon$, then  $\Vert S_{Q(x,\xi)}\tld\lambda_n(x,\xi)-S_{Q(x,\xi')}\tld\lambda_n(x,\xi')\Vert$ is also small. 
We use the same kind of arguments: since we have $\Vert \tld\lambda_n(x,\xi)-\tld\lambda_n(x,\xi')\Vert\leqslant \varepsilon$, we know that $\tld\lambda_n(x,\xi)|_{\tld Q(x,\xi)\setminus \tld Q(x,\xi')}\leqslant\varepsilon$, and therefore $S_{Q(x,\xi)}(\tld\lambda_n(x,\xi)|_{Q(x,\xi)\cap Q(x,\xi')})\leqslant \varepsilon$.
The same thing is of course valid for $\lambda_n(x,\xi')$.

Furthermore, we see that $\vert\tld\lambda_n(x,\xi)-\tld\lambda_n(x,\xi)\vert(Q(x,\xi)\cap Q(x,\xi'))\leqslant \varepsilon$, so that 
$$\Vert S_{Q(x,\xi)\cap Q(x,\xi')}(\lambda_n(x,\xi)|_{Q(x,\xi)\cap Q(x,\xi')}-\lambda_n(x,\xi')|_{Q(x,\xi)\cap Q(x,\xi')}
)\Vert\leqslant\varepsilon,$$

so finally $\Vert\lambda_n(x,\xi)-\lambda_n(x,\xi')\Vert\leqslant 3\varepsilon$, which proves the continuity of $\lambda_n$ in $\xi$, and hence the continuity of $\lambda_n$ since $\Rsph(\Sigma)$ is discrete.\\

The only thing left to do is to force $\lambda_n$ to be $W$-equivariant instead of $W_0$-equivariant. To do so, define  $$\mu_n(x,\xi)=\frac {|W_0|}{|W|} \sum_{w\in W/W_0} w\lambda_n(w^{-1}x,w^{-1}\xi).$$

As $\lambda$ is $W_0$-equivariant, $w\lambda_n(w^{-1}x,w^{-1}\xi)$ does not depend of the choice of $w$ in a class of $W/W_0$, so that $\mu_n$ is well-defined. Moreover, it is easy to check that $\mu_n$ is $W$-equivariant. 
We also have that
$$\lim_{n\to+\infty}\Vert\lambda_n(w^{-1}x,w^{-1}\xi)-\lambda_n(w^{-1}x',w^{-1}\xi)\Vert= 0$$ 
uniformly in $\xi$, so $\lim_n\Vert w.\lambda_n(w^{-1}x,w^{-1}\xi)-w.\lambda_n(w^{-1}x',w^{-1}\xi)\Vert=0$, and in the end $$\lim_{n\to\infty}\Vert\mu_n(x,\xi)-\mu_n(x',\xi)\Vert=0$$ 
uniformly in $\xi$.
Finally, we have 
$$\Supp(w\mu_n(w^{-1}x,w^{-1}\xi))\subset wQ(w^{-1}x,w^{-1}\xi)=Q(x,\xi),$$
 which proves (i).
\end{proof}

\begin{remark}\label{position} Let $y$ be a spherical residue in $Q(x,\xi)$, and let $\widetilde y =(y_1,\dots,y_l)$ be its image in $\Rsph(T_1)\times\dots\times \Rsph(T_l)$. Note that $\mu_n(x,\xi)(y)$ only depends on the position of $y$ towards $(x,\xi)$. Indeed, in each tree $T_i$, the value $\lambda_n^i(x_i,\xi_i)(y_i)$ only depends on the distance between $x_i$ and $y_i$. Consequently, $\widetilde\lambda_n(x,\xi)(\widetilde y)$ only depends on the position (defined in Definition \ref{defposition}) of $y$ towards $(x,\xi)$. Then it is also the case of $S(x,\xi)\tld\lambda_n(x,\xi)(y)$. Finally, $w\lambda_n(w^{-1}x,w^{-1}\xi)(\widetilde y)$ only depends on the position of $w^{-1}y$ towards $(w^{-1}x,w^{-1}\xi)$, that is on the position of $y$ towards $(x,\xi)$.
\end{remark}

\section{Amenability for buildings}\label{section:amenImm}
In this section, $X$ is a  building, $W$ its Weyl group and $G$ its automorphism group.

\begin{theorem}\label{immeuble}
There exists a sequence of continuous $G$-equivariant maps $$\mu_n:\Rsph(X)\times \Csph(X)\to \PR(\Rsph(X))$$ such that $$\lim_{n\to+\infty}\Vert\mu_n(x,\xi)-\mu_n(y,\xi)\Vert=0$$
 uniformly on $\Csph(X)$, for all spherical residues $x$ and $y$.
\end{theorem}

\begin{proof}
 \vspace{.5cm}
{\bf Construction of $\mu_n$}.
Let $A_0$ be an apartment. We choose some trees $T_1,\dots,T_l$ as in Section \ref{arbres}, so that $A$ is embedded into the product $T_1\times\dots\times T_l$. Let $A$ be another apartment. Assume first that $A_0\cap A$ contains a chamber $x$. Let $\rho$ be the retraction onto $A_0$ centered at $x$. Then $\rho_{A'}$ is an isomorphism, and we carry the choice of trees for $A_0$ into a choice of trees for $A$ \textit{via} the retraction $\rho$. This choice of trees does not depend on the choice of $x$ in $A\cap A_0$.

If $A'$ is another apartment, we know that for any $x\in A'$, $y\in A_0$, there is an apartment $A$ containing $x$ and $y$. We can carry in the same way the choice of trees we made on $A$ to $A'$, using the retraction $\rho_{x,A}$. Again, the choice we make does not depend on the choice of $x\in A'\cap A$ and $y\in A_0\cap A$. Furthermore, it does not depend on the choice of $A$ either. Indeed, let $A''$ be another apartment. Assume first that $A''\cap A$ contains some chamber $c$. Then we have a commutative diagram:

 $$
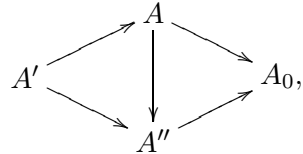
\xymatrix @!0 @C=4pc @R=2pc{ & A \ar[dd]  \ar[rd] \\ A' \ar[ru] \ar[rd] & & A_0,\\ & A''\ar[ru] }$$

where all the arrows are given by restriction centered at a chamber in the intersection of the apartments. It follows that the choice of trees given by $A'$ or $A''$ is the same. Now, if $A''\cap A$ is empty, then there is an apartment $A'''$ containing some chamber $x\in A'\cap A$ and $y\in A_0\cap A''$. In view of the discussion above, the choice of trees in $A'$ given by $A'''$ is the same as the one given by $A''$, and also the same as the one given by $A$. Hence $A''$ and $A$ give the same choice of trees, which proves that this choice only depends on the initial choice of trees we made on $A_0$.

% Let $A$ be an apartment. Let$N_A$ be the stabilizer of $A$ in $G$, and $H_A$ its pointwise stabilizer. Then $N_A/H_A$ is isomorphic to $W$. We choose a normal, torsion-free, finite index subgroup $W_0$ of $W$. As in Section \ref{arbres}, this gives us an imbedding of $A$ into a product of trees $T_1\times\dots\times T_l$.
% We carry this choice in all the apartments of $X$ \textit{via} elements of $G$. In other words, the set of trees on the apartment $gA$ is the set of all the $gT_1,\dots,gT_l$. If $g\in G$ stabilizes $A$, as $G$ is type-preserving, $g$ acts on $A$ as an element of $W$, and therefore stabilizes the set of trees $\{T_1,\dots,T_l\}$. Consequently, this choice of trees in $gA$ is well-defined: it does not depend on the choice of $g$ sending $A$ on $gA$. Moreover, the trees $gT_1,\dots,gT_l$ could also have been defined using the group $gW_0g^{-1}$ acting on $gA$, and consequently satisfy the same properties as in \ref{arbres}.

From an apartment $A$ endowed with a choice of trees $T_1,\dots,T_l$, we get a continuous map $\lambda_n^A:\Rsph(A)\times \Csph(A)\to \PR(\Rsph(A))$, as in Proposition \ref{appt}. Furthermore, if some apartments $A$ and $A'$ are such that there exists $x\in X$ with $Q(x,\xi)\subset A\cap A'$, in restriction to $Q(x,\xi)$, the trees we have chosen for $A$ and $A'$ are the same. More precisely, the positions of a residue $y$ in $Q(x,\xi)$ towards $(x,\xi)$ are the same, whether they are determined in $A$ or in $A'$. Moreover, by remark \ref{position}, $\lambda_n(x,\xi)(B)$ only depends on the position of chambers of $B$ in each of the trees $T_1,\dots, T_l$ towards $(x,\xi)$. Consequently, if $B$ is a subset of $X$, we can see that $\lambda_n^A(B\cap Q(x,\xi))$ does not depend on the choice of the  $A$  containing $Q(x,\xi)$.

 Thus, we can define $\mu_n(x,\xi)$ by $\mu_n(x,\xi)(B)=\lambda_n^A(x,\xi)(B\cap Q(x,\xi))$, for any apartment $A$ containing $Q(x,\xi)$. 

\vspace{.5cm}
{\bf Equivariance of $\mu_n$}. 

Let $g\in G$ and let $F$ be a subset of $\Rsph(X)$. Let $x\in\Rsph(X)$ and $\xi\in\Csph(X)$. Let $A$ be an apartment containing $Q(x,\xi)$. We have $g.\mu_n(x,\xi)(F)=\lambda_n^A(x,\xi)(g^{-1}F\cap Q(x,\xi))=g.\lambda_n^A(x,\xi)(F\cap Q(gx,g\xi))$. But $g.\lambda_n^A(x,\xi)$ is precisely the measure supported on $Q(gx,g\xi)$ defined by Proposition \ref{appt}, in the apartment $gA$, with respect to the trees that are the images by $g$ of the trees we have chosen in $A$. Thus, all we have to do is to prove that the choice of trees we made on the apartments of $X$ is invariant with respect to $G$.
Let $A_1=gA_2$ and $A_2$ be an apartment such that each of the intersections $A_2\cap A_1$ and $A_2\cap A_0$ contain a chamber. Since $G$ is type-preserving, there exists an element $w\in W$ such that the following diagram is commutative:

$$
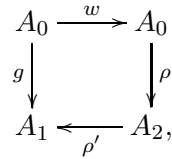
\xymatrix {A_0 \ar[r]^w \ar[d]_g& A_0\ar[d]^\rho \\ A_1 &A_2,\ar[l]^{\rho'} }
$$
where $\rho$ is the retraction onto $A_2$ centered at any chamber in any chamber in $A_2\cap A_0$, and $\rho'$ is the retraction onto $A_1$ centered at a chamber in $A_1\cap A_2$. Hence the images of the trees in $A_0$ by $g$ are the same as the images of the trees by $\rho'\circ\rho\circ w$, which are, by definition, the trees in $A_1$. Thus the system of trees in the apartments of $X$ is $G$-invariant.

Consequently, the measure $g.\lambda_n^A(x,\xi)$ is equal to $\lambda_n^{gA}(gA,g\xi)$, and we get $g.\mu_n(x,\xi)(F)=\mu_n(gx,g\xi)(F)$.

\vspace{.5cm}
{\bf Limit of $\Vert\mu_n(x,\xi)-\mu_n(y,\xi)\Vert$}.

 Let $x$ and $y$ be chambers in $X$. We want to estimate $\Vert\mu_n(x,\xi)-\mu_n(y,\xi)\Vert$. The measure $|\mu_n(x,\xi)-\mu_n(y,\xi)|$ is supported in $Q(x,\xi)\cup Q(y,\xi)$. This set can be divided into three disjoint parts: $Q(x,\xi)\setminus Q(y,\xi)$, $Q(x,\xi)\cap Q(y,\xi)$ and $Q(y,\xi)\setminus Q(z,\xi)$. As the support of $\mu_n(x,\xi)$ (respectively of $\mu_n(y,\xi)$) is included in $Q(x,\xi)$ (resp. $Q(y,\xi)$), we have:

\begin{multline*}
\Vert\mu_n(x,\xi)-\mu_n(y,\xi)\Vert=\mu_n(x,\xi)(Q(x,\xi)\setminus Q(y,\xi))+|\mu_n(x,\xi)-\mu_n(y,\xi)|(Q(x,\xi)\cap Q(y,\xi))\\
+\mu_n(y,\xi)(Q(y,\xi)\setminus Q(x,\xi)).
\end{multline*}
 
 We prove that all these three terms tend to $0$.

 Let $A$ be an apartment containing $Q(x,\xi)$ and $A'$ an apartment containing $Q(y,\xi)$. We know by Proposition \ref{intersect} that there exists a chamber in $A\cap A'$.
 Let $\rho$ be the retraction onto $A$ centered at $x$. As there exists a sequence of chambers converging to $\xi$ and contained in $A\cap A'$, we see that  $\rho$ also fixes $\xi$. Let $y'=\rho(y)$. 
By definition of the trees in $A'$, we know that $\rho.\lambda_n^{A'}(y,\xi)=\lambda_n^{A}(\rho(y),\xi)$.

 Since $Q(x,\xi)\cap Q(y,\xi)\subset A\cap A'$, we see that $Q(x,\xi)\cap Q(y,\xi)=Q(x,\xi)\cap Q(y',\xi)$. 
We deduce that $Q(x,\xi)\setminus Q(y,\xi)=Q(x,\xi)\setminus Q(y',\xi)$. 
Hence we have 
$$\mu_n(x,\xi)(Q(x,\xi)\setminus Q(y,\xi))=\lambda_n^{A}(x,\xi)(Q(x,\xi)\setminus Q(y,\xi))=\lambda_n^{A}(Q(x,\xi)\setminus Q(y',\xi)).$$

 Now, in the same way as above, $Q(x,\xi)\cup Q(y',\xi)$ can be divided into three parts, and we can write:

\begin{multline}\label{decomp}
\Vert\lambda^A_n(x,\xi)-\lambda^A_n(y',\xi)\Vert=\lambda^A_n(x,\xi)(Q(x,\xi)\setminus Q(y',\xi))\\+|\lambda^A_n(x,\xi)-\lambda^A_n(y',\xi)|(Q(x,\xi)\cap Q(y',\xi))
+\lambda^A_n(y',\xi)(Q(y',\xi)\setminus Q(x,\xi)).
\end{multline}
 
 Consequently, we have
$$\lambda^A_n(x,\xi)(Q(x,\xi)\setminus Q(y',\xi))\leqslant \Vert\lambda^A_n(x,\xi)-\lambda^A_n(y',\xi)\Vert.$$
 As this quantity tends to $0$ uniformly in $\xi$, there is a sequence of numbers $\eps_n$ only depending on $x$ and $y$ such that $\Vert\lambda^A_n(x,\xi)-\lambda^A_n(y',\xi)\Vert\leqslant \eps_n$ and $\lim_n\eps_n=0$. Thus, we have $\mu_n(x,\xi)(Q(x,\xi)\setminus Q(y,\xi))\leqslant\eps_n$.
 
 By switching $x$ and $y$, we see that we have also  $\mu_n(y,\xi)(Q(y,\xi)\setminus Q(x,\xi))\leqslant\eps_n$. So, only the second term is left, that is $|\mu_n(x,\xi)-\mu_n(y,\xi)|(Q(x,\xi)\cap Q(y,\xi))$. As $Q(x,\xi)\cap Q(y,\xi)\subset A\cap A'$, we can write $$|\mu_n(x,\xi)-\mu_n(y,\xi)|(Q(x,\xi)\cap Q(y,\xi))=|\lambda^A_n(x,\xi)-\lambda^{A'}_n(y,\xi)|(Q(x,\xi)\cap Q(y,\xi)). $$
 
 As before, this term is better evaluated in $A\cap A'$. More precisely, we have:
 
 \begin{eqnarray*}
 |\lambda^A_n(x,\xi)-\lambda^{A'}_n(y,\xi)|(Q(x,\xi)\cap Q(y,\xi))
&=&|\lambda^A_n(x,\xi)-\lambda^{A'}_n(y,\xi)|(\rho^{-1}(Q(x,\xi)\cap Q(y,\xi)))\\
 &=&|\rho.\lambda^A_n(x,\xi)-\rho.\lambda^{A'}_n(y,\xi)|(Q(x,\xi)\cap Q(y,\xi))\\
 &=&|\lambda^A_n(x,\xi)-\lambda^A_n(y',\xi)|(Q(x,\xi)\cap Q(y,\xi)),
 \end{eqnarray*}
 since $\rho$ fixes $x$ and $\xi$.
 
As we have already seen, we have $Q(x,\xi)\cap Q(y,\xi)=Q(x,\xi)\cap Q(y',\xi)$. Consequently, the last line is equal to
 $|\lambda^A_n(x,\xi)-\lambda^A_n(y',\xi)|(Q(x,\xi)\cap Q(y',\xi))$, which is also less that $\eps_n$ by \ref{decomp}.

Consequently, we have $\Vert\mu_n(x,\xi)-\mu_n(y,\xi)\Vert\leqslant 3\eps_n$, thus $\Vert\mu_n(x,\xi)-\mu_n(y,\xi)\Vert$ tends to $0$ uniformly in $\xi$.

\vspace{.5cm}
{\bf Continuity of $\mu_n$}. 
As  $\Rsph(X)$ is a discrete set, we have to prove that $\mu_n$ is continuous with respect to its second argument.

By Proposition \ref{induction} , we see that if $\xi_k,\xi\in\Csph(X)$ are such that $\lim_k\xi_k=\xi$, then $\lim_k Q(x,\xi_k)=Q(x,\xi)$, with the pointwise topology.

Let $F$ be a finite subset of $\Rsph(X)$. For $n$ large enough, we have seen that $Q(x,\xi_n)\cap F=Q(x,\xi)\cap F$. Moreover, 
 $\mu_n(x,\xi_k)(F)=\lambda_n(x,\xi_k)(F\cap Q(x,\xi_k))$. Let $A_k$ be an apartment containing $Q(x,\xi_k)$ and $A$ an apartment containing $Q(x,\xi)$. By Remark \ref{position}, $\lambda_n(x,\xi_k)(F\cap Q(x,\xi_k)$ (resp.  $\lambda_n(x,\xi)(F\cap Q(x,\xi))$) only depends on the position of the residues in $F\cap Q(x,\xi)$ towards $x$ in the trees chosen relatively to $A_k$ (resp. to $A$). Consequently, these two quantities are equal, in other words, $\mu_n(x,\xi_k)(F)=\mu_n(x,\xi)(F)$. 

Hence we have $\lim_{k\to +\infty}\mu_n(x,\xi_k)=\mu_n(x,\xi)$ for the pointwise topology. But, in restriction to $\PR(\Rsph(X))$, the pointwise topology and the norm topology are the same. So $\mu_n$ is continuous with respect to its second variable, and hence is continuous, since $\Rsph(X)$ is discrete.
\end{proof}

\begin{corollary}\label{fin}
Let $X$ be any building, and $\Gamma$ a locally compact subgroup of $\Aut(X)$ such that stabilisers in $\Gamma$ of spherical residues in $X$ are compact. Then the action of $\Gamma$ on the combinatorial compactification of $X$ is amenable.
\end{corollary}

\begin{proof} We know there are some maps $\mu_n$ as in Theorem \ref{immeuble}. The result follows from Corollary \ref{cor:moyennable}.
\end{proof}

Note that the condition that stabilisers of spherical residues are compact is equivalent to the properness of the action of $\Gamma$ on the Davis-Moussong realisation of $X$, since the stabiliser of a point is the stabiliser of the minimal facet containing it, which corresponds to a spherical residue.

\bibliographystyle{alpha}
\bibliography{bibmoyenn}
\end{document}